\documentclass{birkau}

\usepackage{amsmath,amssymb,url}
\usepackage{enumitem} 
\usepackage[T1]{fontenc}
\usepackage{cite}

\usepackage{tikz}
\usetikzlibrary{positioning}  

\numberwithin{equation}{section}

\theoremstyle{plain}
\newtheorem{thm}{Theorem}[section]
\newtheorem{lem}[thm]{Lemma}
\newtheorem{cor}[thm]{Corollary}
\newtheorem{pro}[thm]{Proposition}

\theoremstyle{definition}

\newcommand{\bp}{\mathbf{p}}
\newcommand{\bq}{\mathbf{q}}

\newcommand{\bu}{\mathbf{u}}
\newcommand{\bv}{\mathbf{v}}
\newcommand{\bw}{\mathbf{w}}

\begin{document}
\title[The Continuum of varieties]
{The Interval $[\mathsf{V}(S_7),\mathsf{V}(B_2^1)]$ of Semiring Varieties Has the Cardinality of the Continuum}

\author[Z. D. Gao]{Zi Dong Gao}
\address{School of Mathematics\\
Northwest University\\Xi'an 710127\\Shaanxi\\P.R. China}
\email{zidonggao@yeah.net}

\author[M. M. Ren]{Miao Miao Ren}
\address{School of Mathematics\\
Northwest University\\Xi'an 710127\\Shaanxi\\P.R. China}
\urladdr{https://math.nwu.edu.cn/info/1269/6520.htm}
\email{miaomiaoren@yeah.net}

\author[M. Y. Yue]{Meng Ya Yue}
\address{School of Mathematics\\
Northwest University\\Xi'an 710127\\Shaanxi\\P.R. China}
\email{myayue@yeah.net}



\begin{abstract}
We prove that the interval $[\mathsf{V}(S_7),\mathsf{V}(B_2^1)]$
in the lattice of additively idempotent semiring (ai-semiring) varieties has the cardinality of the continuum,
where $S_7$ is the smallest nonfinitely based ai-semiring (a three-element algebra),
and $B_2^1$ is the ai-semiring whose multiplicative reduct is the six-element Brandt monoid.
\end{abstract}

\maketitle

\section{Introduction}\label{sec:intro}
A \emph{variety} is a class of algebras closed under taking subalgebras, homomorphic images, and arbitrary direct products.
By Birkhoff's theorem, a class of algebras is a variety if and only if it is an \emph{equational class},
that is, it consists of all algebras satisfying a certain set of identities.

For a class $\mathcal{K}$ of algebras,
we write $\mathsf{V}(\mathcal{K})$ for the variety generated by $\mathcal{K}$,
that is, the smallest variety containing $K$.
In particular, $\mathsf{V}(A)$ denotes the variety generated by a single algebra $A$.
A variety is \emph{finitely based} if it can be defined by a finite set of identities;
otherwise, it is \emph{nonfinitely based}.
Correspondingly,
an algebra $A$ is finitely based or nonfinitely based if the variety $\mathsf{V}(A)$
is finitely based or not.
The finite basis problem for a class of algebras,
one of the most important problems in universal algebra,
concerns the classification of its members according to whether they are finitely based.

For a variety $\mathcal{V}$,
the collection of all its subvarieties forms a complete lattice with respect to inclusion;
we denote it by $\mathcal{L}(\mathcal{V})$.
If $\mathcal{W}$ is a subvariety of $\mathcal{V}$,
then the interval $[\mathcal{W},\mathcal{V}]$, consisting of all varieties between $\mathcal{W}$ and $\mathcal{V}$,
is a complete sublattice of $\mathcal{L}(\mathcal{V})$.

Understanding the structure of the subvariety lattice $\mathcal{L}(\mathcal{V})$ is a fundamental problem in the study of varieties.
A central aspect of this problem is determining the size of intervals within the lattice.
Such descriptions reveal not only the internal complexity of $\mathcal{V}$,
but also often reflect important algebraic and logical properties of the algebras it contains.

An \emph{additively idempotent semiring} (ai-semiring for short)
is an algebra $(S, +, \cdot)$ equipped with two binary operations $+$ and $\cdot$
such that the additive reduct $(S, +)$ is a commutative idempotent semigroup,
the multiplicative reduct $(S, \cdot)$ is a semigroup and the distributive laws
\[
x(y+z)\approx xy+xz,\quad (x+y)z\approx xz+yz
\]
hold.
Such algebras are ubiquitous in mathematics and find applications in diverse areas such as
algebraic geometry~\cite{cc}, tropical geometry~\cite{ms}, information science~\cite{gl}, and theoretical computer science~\cite{go}.

Let $S$ be an ai-semiring. Then the relation $\leq$ on $S$ defined by
\[
a \leq b \Leftrightarrow a+b=b
\]
is a partial order such that $(S, \leq)$ forms an upper semilattice, where the supremum of any two elements $a$ and $b$ is $a+b$.
Consequently, the additive reduct $(S, +)$ is uniquely determined by this semilattice order.
It is therefore often convenient to visualize the addition via the Hasse diagram of $(S, \leq)$.
Moreover, the order $\leq$ is readily seen to be compatible with multiplication,
which explains such an algebra is also termed a \emph{semilattice-ordered semigroup}.

Over the past two decades, the finite basis problem and other variety-theoretic properties
for ai-semirings have been intensively studied and well developed
(see~\cite{dol07, dol09, gjrz, gpz, jrz, pas05, rlyc, rlzc, rjzl, rzv, rz16, rzw, sr, vol21, yrzs, zrc}).
In particular,
Pastijn et al.~\cite{gpz, pas05} established that there are precisely $78$ ai-semiring
varieties satisfying the identity $x^2\approx x$, all of which are finitely based.
Ren et al.~\cite{rz16, rzw} later proved that there are precisely $179$ ai-semiring
varieties satisfying the identity $x^3\approx x$, which are likewise all finitely based.
Recently, Volkov et al.\cite{rzv} showed that for any integer $n\geq 4$,
there are $2^{\aleph_0}$ distinct varieties satisfying $x^n \approx x$,
most of which are nonfinitely based.

\begin{table}[ht]
\caption{The Cayley tables of $S_7$} \label{tb24111401}
\begin{tabular}{c|ccc}
$+$      &$0$&$a$&$1$\\
\hline
$0$ &$0$&$0$&$0$\\
$a$      &$0$&$a$&$0$\\
$1$      &$0$&$0$&$1$\\
\end{tabular}\qquad
\begin{tabular}{c|ccc}
$\cdot$  &$0$&$a$&$1$\\
\hline
$0$      &$0$&$0$&$0$\\
$a$      &$0$&$0$&$a$\\
$1$      &$0$&$a$&$1$\\
\end{tabular}
\end{table}

On the other hand,
Zhao et al.~\cite{zrc} showed that all ai-semirings of order at most three are finitely based,
with the possible exception of the semiring $S_7$; see Table~\ref{tb24111401} for its Cayley tables.
Jackson et al.~\cite{jrz} later established that $S_7$ itself is nonfinitely based.
In fact, they showed that $S_7$ can transmit this property to many other finite ai-semirings.
Such examples include finite flat semirings whose varieties contain $S_7$,
and the ai-semiring $B^1_2$ whose multiplicative reduct is the six-element Brandt monoid.

Recall that the elements of $B^1_2$ are the matrices
\[
\begin{array}{cccccc}
\begin{pmatrix}
0 & 0 \\
0 & 0
\end{pmatrix} &
\begin{pmatrix}
1 & 0 \\
0 & 1
\end{pmatrix} &
\begin{pmatrix}
0 & 1 \\
0 & 0
\end{pmatrix} &
\begin{pmatrix}
0 & 0 \\
1 & 0
\end{pmatrix} &
\begin{pmatrix}
1 & 0 \\
0 & 0
\end{pmatrix} &
\begin{pmatrix}
0 & 0 \\
0 & 1
\end{pmatrix} \\
0 & 1 & a & b & ab & ba
\end{array}
\]
with multiplication given by ordinary matrix multiplication.
Its additive reduct is determined by the Hasse diagram in Figure~\ref{fig1}.
Perkins~\cite{perkins1969} showed that the multiplicative reduct of $B^1_2$ is nonfinitely based;
indeed, it was the first example of a nonfinitely based finite semigroup,
and is sometimes referred to as Perkins' semigroup.

One can readily verify that $B^1_2$ contains a copy of $S_7$;
hence $\mathsf{V}(S_7)$ is a subvariety of $\mathsf{V}(B^1_2)$,
and the interval $[\mathsf{V}(S_7), \mathsf{V}(B_2^1)]$ is well-defined.
We remark that Volkov~\cite{vol21} independently resolved the finite basis problem for $B^1_2$ by a different method.

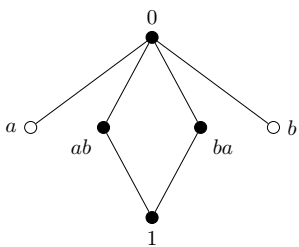
\begin{figure}[h]
\centering
\scalebox{0.8}{%
\begin{tikzpicture}[
    node distance=1.5cm and 1.5cm,
    solidnode/.style={circle, draw=black, fill=black, inner sep=2pt, minimum size=4pt},
    hollownode/.style={circle, draw=black, fill=white, inner sep=2pt, minimum size=4pt}
]

\node[solidnode, label=below:1] (one) at (0,0.5) {};

\node[hollownode, label=left:{$a$}] (e12) at (-2,2) {};
\node[solidnode, label=below left:{$ab$}] (e11) at (-0.8,2) {};
\node[solidnode, label=below right:{$ba$}] (e22) at (0.8,2) {};
\node[hollownode, label=right:{$b$}] (e21) at (2,2) {};

\node[solidnode, label=above:0] (zero) at (0,3.5) {};

\draw (one) -- (e11);
\draw (one) -- (e22);

\draw (e12) -- (zero);
\draw (e11) -- (zero);
\draw (e22) -- (zero);
\draw (e21) -- (zero);

\end{tikzpicture}%
}
\caption{The additive order of $B_2^1$}
\label{fig1}
\end{figure}

Over recent years,
the finite ai-semirings $S_7$ and $B_2^1$ have become pivotal objects in the study of ai-semiring varieties.
Gusev and Volkov~\cite{gv2301, gv2302, gv2501} identified mild conditions
under which a finite ai-semiring whose variety contains $B_2^1$ is nonfinitely based.
Gao et al.~\cite{gjrz} proved that $\mathsf{V}(S_7)$ itself contains $2^{\aleph_0}$ distinct subvarieties,
of which only six are finitely based.
More recently, they~\cite{gjrz2} further proved
that every variety in the interval $[\mathsf{V}(S_7), \mathsf{V}(B_2^1)]$ is nonfinitely based.

These developments are situated within a broader historical context.
Dolinka~\cite{dol08} previously showed that the interval
$[\mathsf{V}(\Sigma_7), \mathsf{V}(\mathbf{Rel}(2))]$ also has cardinality $2^{\aleph_0}$,
where $\Sigma_7$ is the $7$-element ai-semiring obtained from $B_2^1$ by adding a new element
that serves as both the additive least element and the multiplicative zero element,
and $\mathbf{Rel}(2)$ is the semiring of all binary relations on a $2$-element set.
To the best of our knowledge,
$\mathbf{Rel}(2)$ is the first example of a finite ai-semiring whose variety has $2^{\aleph_0}$ distinct subvarieties,
and $\Sigma_7$ itself was the first example of a nonfinitely based finite ai-semiring (see~\cite{dol07}).
Note that the interval $[\mathsf{V}(B_2^1), \mathsf{V}(\Sigma_7)]$ also has cardinality $2^{\aleph_0}$ (see~\cite{gjrz2}).

These results collectively paint a picture of intricate complexity within ai-semiring varieties.
However, a fundamental question concerning the interval $[\mathsf{V}(S_7), \mathsf{V}(B_2^1)]$
has remained unanswered: while all its members are known to be nonfinitely based,
what is its precise cardinality?
Specifically, does it also attain the maximal possible cardinality of $2^{\aleph_0}$,
or is it strictly smaller?
This paper addresses this open question. Our main contribution is the following theorem.

\begin{thm}\label{maintheorem}
The interval $[\mathsf{V}(S_7),\mathsf{V}(B_2^1)]$ has cardinality $2^{\aleph_0}$.
\end{thm}

For an ai-semiring $S$,
we denote by $\mathsf{V}_s(S)$ the semigroup variety generated by the multiplicative semigroup reduct of $S$.
We obtain a parallel result for semigroup varieties.

\begin{thm}\label{thm26010401}
The interval \([\mathsf{V}_s(S_7),\mathsf{V}_s(B_2^1)]\) has cardinality $2^{\aleph_0}$.
\end{thm}

The following proposition encapsulates a standard method widely used in the literature to
prove that a variety has $2^{\aleph_0}$ distinct subvarieties (see~\cite{dol08, jac:uncoutably, jaclee, gla, gus, rjzl, tra}).
While the underlying idea is frequently applied implicitly to construct large intervals in subvariety lattices,
it is seldom stated explicitly as a formal, standalone result;
we present it here for the reader's convenience and later reference.
For a set $I$, let $\mathcal{P}(I)$ denote the lattice of all subsets of $I$ ordered by inclusion.

\begin{pro}\label{semantics}
Let $\mathcal{V}$ be a variety and let $I$ be a countably infinite set.
Then the following statements are equivalent:
\begin{enumerate}[label=\textup{(\alph*)}]
\item The lattice $\mathcal{P}(I)$ embeds into the subvariety lattice $\mathcal{L}(\mathcal{V})$.

\item There exists a family ${(A_i)}_{i\in I}$ of algebras in $\mathcal{V}$
and a family $\{\sigma_i\}_{i \in I}$ of identities such that for all $i, j \in I$,
$A_i$ satisfies $\sigma_j$ if and only if $i \neq j$.
\end{enumerate}
When either (hence both) condition holds,
the lattice $\mathcal{L}(\mathcal{V})$ has cardinality $2^{\aleph_0}$;
moreover, if $\mathcal{W}$ is a subvariety of $\mathcal{V}$ that is contained in $\mathsf{V}(A_i)$ for all $i\in I$,
then the interval $[\mathcal{W}, \mathcal{V}]$ also has cardinality $2^{\aleph_0}$.
\end{pro}
\begin{proof}
Suppose that $(a)$ holds.
Then there exists an embedding mapping $\varphi\colon \mathcal{P}(I) \to \mathcal{L}(\mathcal{V})$.
This implies that $\varphi$ and $\varphi^{-1}$ are both order-preserving.
For each $i\in I$,
let $A_i$ be a free algebra in the variety $\varphi(\{i\})$ on a countably infinite set of free generators.
By a standard fact (see~\cite[Lemma 11.8]{bs}), $\varphi(\{i\}) = \mathsf{V}(A_i)$.
Since $\varphi(\{i\})$ is not contained in $\varphi(I\setminus\{i\}) = \mathsf{V}(\{A_j \mid j \in I\setminus\{i\}\})$,
there exists an identity $\sigma_i$ that holds in every $A_j$ with $j \neq i$ but fails in $A_i$.
This establishes the condition $(b)$.

Conversely, assume that $(b)$ is true. Consider the mapping
\[
\varphi\colon \mathcal{P}(I) \to \mathcal{L}(\mathcal{V}), \quad M \mapsto \mathsf{V}(\{A_m \mid m \in M\}).
\]
By assumption,
for any $M \subseteq I$ and any $i\in I$,
the variety $\varphi(M)$ satisfies $\sigma_i$ if and only if $i \notin M$.
Consequently, for any $M, N\subseteq I$ we have that $M\subseteq N$ if and only if $\varphi(M)\subseteq \varphi(N)$.
Thus $\varphi$ is a lattice embedding mapping, and so $\mathcal{P}(I)$ embeds into $\mathcal{L}(\mathcal{V})$.

Finally, suppose that the condition $(a)$ or $(b)$ holds,
Then $\mathcal{P}(I)$ embeds into $\mathcal{L}(\mathcal{V})$.
Since $I$ is countably infinite,
it follows immediately that $\mathcal{P}(I)$ has cardinality $2^{\aleph_0}$.
Consequently, $\mathcal{L}(\mathcal{V})$ has cardinality $2^{\aleph_0}$.
Moreover, if $\mathcal{W}$ is a subvariety of $\mathcal{V}$ that is contained in $\mathsf{V}(A_i)$ for all $i\in I$,
then the mapping $\varphi$ above
can be viewed as a mapping from $\mathcal{P}(I)$ to the interval $[\mathcal{W}, \mathcal{V}]$,
whence the interval also has cardinality $2^{\aleph_0}$.
\end{proof}

\section{Preliminaries}
In this section, we establish the groundwork and provide the necessary definitions for the proof of our main theorem.

A \emph{flat semiring} is an ai-semiring $S$ whose multiplicative reduct has a zero $0$
and whose addition satisfies $a+b =0$ for all distinct $a, b \in S$.
Jackson et al.~\cite[Lemma 2.2]{jrz} observed that a semigroup with zero $0$
becomes a flat semiring if and only if it is $0$-cancellative, that is,
$ab=ac\neq 0$ implies $b=c$ and $ab=cb\neq 0$ implies $a=c$ for all $a, b, c\in S$.

The following algebras form an important class of flat semirings.
Let $W$ be a nonempty subset of a free semigroup,
and let $S(W)$ denote the set of all nonempty subwords of words in
$W$ together with a new symbol $0$. Define a binary operation $\cdot$ on $S(W)$ by the rule
\begin{equation*}
\bu\cdot \bv=
\begin{cases}
\bu\bv& \text{if }~\bu\bv\in S_c(W)\setminus \{0\}, \\
0& \text{otherwise.}
\end{cases}
\end{equation*}
Then $(S(W), \cdot)$ forms a commutative semigroup with the zero element $0$.
It is easy to verify that $(S(W), \cdot)$ is $0$-cancellative and so $S(W)$ becomes a flat semiring.
In particular, if~$W$ consists of a single word $\bw$ we shall write $S(W)$ as $S(\bw)$.
If we allow the empty word $\varepsilon$ in this construction, then
the multiplicative reduct is a monoid.
The notation $M(W)$ and $M(\bw)$ are used in this case.
If $a$ is a letter, then $M(a)$ is isomorphic to $S_7$.
In particular, $M(\varepsilon)$, which is denoted by $M_2$ in \cite{sr}, can be embedded into $S_7$.

Let $X^+$ denote the free semigroup over a countably infinite set $X$ of variables.
By distributivity, all ai-semiring terms over $X$ can be expressed as a finite sum of words from words in $X^+$.
An \emph{ai-semiring identity} over $X$ is an
expression of the form
\[
\bu\approx \bv,
\]
where $\bu$ and $\bv$ are ai-semiring terms over $X$.
Let $\bu\approx \bv$ be an ai-semiring identity over an alphabet $\{x_1, x_2, \ldots, x_n\}$.
We say that an ai-semiring \emph{$S$ satisfies $\bu\approx \bv$} or \emph{$\bu\approx \bv$ holds in $S$},
if $\bu(a_1, a_2, \ldots, a_n)=\bv(a_1, a_2, \ldots, a_n)$ for all $a_1, a_2, \ldots, a_n\in S$,
where $\bu(a_1, a_2, \ldots, a_n)$ denotes the result of evaluating $\bu$ in $S$
under the assignment $x_i\rightarrow a_i$, and similarly for $\bv(a_1, a_2, \ldots, a_n)$.

Let $\bp$ be a word in $X^+$. Then $c(\bp)$ denotes the set of all variables that occur in $\bp$.
For an ai-semiring term $\bu=\bu_1+\cdots+\bu_n$ with each $\bu_i \in X^+$,
we shall use $c(\bu)$ to denote the set of all variables that occur in $\bu$, that is,
\[
c(\bu)=\bigcup_{1\leq i \leq n}c(\bu_i).
\]
By \cite[Lemma 1.1]{sr}, $M_2$ satisfies an ai-semiring identity $\bu\approx \bv$
if and only if $c(\bu)=c(\bv)$.

For any ai-semiring terms $\bu$ and $\bv$,
let $\bu \preceq \bv$ denote the identity $\bu + \bv \approx \bv$.
We call $\bu \preceq \bv$ an \emph{inequality}.
Let $S$ be an ai-semiring.
We  write $\bu \preceq_{S} \bv$ if $S$ satisfies the identity $\bu \preceq \bv$.
A word $\bw$ is \emph{minimal} for $S$ if
$\bu \preceq_{S} \bw$ implies $\bu = \bw$ for every word $\bu$.
Equivalently,
$\bw$ is minimal for $S$ if and only if it is an \emph{isoterm} for $S$;
that is, $S$ satisfies no nontrivial identity of the form $\bw \approx \bu$ for any ai-semiring term $\bu$.
If $\bw$ is an isoterm for $S$,
then every word obtained from a subword of $\mathbf{w}$ by renaming its letters
is also an isoterm for $S$.

\begin{lem}[{\cite[Lemma 5.6]{rjzl}}]\label{isosw}
Let $\bw$ be a word and let $S$ be an ai-semiring.
Then $\bw$ is an isoterm for $S$
if and only if the flat semiring $S(\bw)$ belongs to $\mathsf{V}(S)$.
\end{lem}

Let $\bu$ and $\bv$ be words.
Then \emph{$\bv$ contains a value of $\bu$} if $\varphi(\bu)$ is a subword of $\bv$ for some substitution $\varphi$;
otherwise, \emph{$\bv$ is $\bu$-free}.
Now suppose that $\bv$ is $\bu$-free. Then the flat semiring $S(\bv)$ satisfies the identities
\begin{equation*}\label{sgvfree}
\bu x \approx x \bu \approx \bu,
\end{equation*}
where $x$ is a variable not occurring in $\bu$.
This identity will be abbreviated as $\bu \approx 0$.
Evidently, the flat semiring $S(\bu)$ does not satisfy the identity $\bu \approx 0$.
We now have the following:
\begin{pro}\label{pro26010201}
Let $\{\bw_i\}_{i\in I}$ be a set of words.
Suppose that the words $\{\bw_i\}_{i\in I}$ are pairwise free.
Then for any $i, j \in I$, the flat semiring $S(\bw_i)$
satisfies the identity $\bw_j \approx 0$ if and only if $i \neq j$.
\end{pro}

Let $S$ be an ai-semiring.
By Propositions~\ref{semantics} and \ref{pro26010201}, together with Lemma~\ref{isosw},
To demonstrate that the variety $\mathsf{V}(S)$ has
$2^{\aleph_0}$ distinct subvarieties,
it suffices to find an infinite set of pairwise free words that are all isoterms for $S$.

\section{The Cardinality of $[\mathsf{V}(S_7),\mathsf{V}(B_2^1)]$}
In this section, we apply Proposition~\ref{semantics} to prove that
$[\mathsf{V}(S_7),\mathsf{V}(B_2^1)]$ contains $2^{\aleph_0}$ distinct varieties.
Throughout, $\mathbb{N}$ denotes the set of positive integers.

Recall that the Zimin word $\mathbf{z}_n$ is defined inductively by
$\mathbf{z}_1 = x_1$ and $\mathbf{z}_{n+1} = \mathbf{z}_n x_{n+1} \mathbf{z}_n$ for all integers $n \geq 1$.
For each $n \in \mathbb{N}$,
let $\bw_n$ denote the word
\begin{equation}
y_0x_1x_2y_0\left(\prod_{i=1}^n(y_ix_{i+2}y_i)\right)y_{n+1}x_{n+3}x_{n+4}y_{n+1}.
\end{equation}
This word, introduced by Sapir and Volkov~\cite{sv},
consists of $n$ disjoint copies of the Zimin word $\mathbf{z}_2$ placed in the central part,
flanked at the beginning and end by two disjoint copies of the word $yx_1x_2y$.
It underlies many constructions in \cite{jac:uncoutably}
and was also used in \cite[Example 5.7]{rjzl} to prove that
the ai-semiring variety $\mathsf{V}(M(abacdc))$ has $2^{\aleph_0}$ distinct subvarieties.

\begin{lem}\label{wfree}
Let $m, n\in \mathbb{N}$.
Then the flat semiring $S(\bw_n)$ satisfies the identity $\bw_m\approx 0$ if and only if $n \neq m$.
\end{lem}
\begin{proof}
If $n \neq m$, then the word $\bw_n$ is $\bw_m$-free.
To see this, observe that each subword of $\bw_n$ of length greater than $1$ occurs exactly once in $\bw_n$.
Consequently, any semigroup substitution mapping $\bw_m$ to a subword of $\bw_n$ must send letters to letters.
Such a letter-to-letter mapping is clearly possible only when $n = m$.
\end{proof}

By swapping the positions of $x_1$ and $x_2$, we obtain a copy of $\bw_n$:
\[
\bw_n' = y_0x_2x_1y_0\left(\prod_{i=1}^n(y_ix_{i+2}y_i)\right)y_{n+1}x_{n+3}x_{n+4}y_{n+1}.
\]

\begin{lem}\label{wmwm'}
Let $m, n\in \mathbb{N}$. Then the flat semiring $S(\mathbf{w}_n)$ satisfies the identity
\begin{equation}\label{eqwmwm'}
\mathbf{w}_m \approx \mathbf{w}_m'
\end{equation}
if and only if $n \neq m$.
\end{lem}
\begin{proof}
It is easy to see that $S(\mathbf{w}_m)$ does not satisfy the identity~\eqref{eqwmwm'}.
If $n\neq m$, then by Lemma~\ref{wfree},
both $\mathbf{w}_m\approx 0$ and $\mathbf{w}_m'\approx 0$ hold in $S(\mathbf{w}_n)$.
Hence $S(\mathbf{w}_n)$ satisfies the identity~\eqref{eqwmwm'}.
\end{proof}

\begin{thm}\label{thm1}
Let $S$ be an ai-semiring.
If $\mathsf{V}(S)$ contains $S_7$ and every word
$\mathbf{w}_n$ $(n\in \mathbb{N})$ is an isoterm for $S$,
then the interval $[\mathsf{V}(S_7), \mathsf{V}(S)]$ has cardinality $2^{\aleph_0}$.
\end{thm}
\begin{proof}
Since $\mathsf{V}(S)$ contains $S_7$,
it follows immediately that $\mathsf{V}(S_7)$ is a subvariety of $\mathsf{V}(S)$,
and so the interval $[\mathsf{V}(S_7), \mathsf{V}(S)]$ is well-defined.
For each $n\in \mathbb{N}$,
because $\mathbf{w}_n$ is an isoterm for $S$,
Lemma~\ref{isosw} implies that the flat semiring $S(\mathbf{w}_n)$
lies in $\mathsf{V}(S)$.
Let $A_n$ denote the direct product of $S_7$ and $S(\mathbf{w}_n)$.
Then the variety $\mathsf{V}(A_n)$ is a member of the interval $[\mathsf{V}(S_7), \mathsf{V}(S)]$.
It is easy to see that $S_7$ satisfies the identity $\mathbf{w}_m \approx \mathbf{w}_m'$ for all $m\geq 1$,
since the multiplicative reduct of $S_7$ is commutative.
By Lemma~\ref{wmwm'}, for any $m, n\in \mathbb{N}$,
the algebra $A_n$ satisfies $\mathbf{w}_m \approx \mathbf{w}_m'$ if and only if $n\neq m$.
Therefore, Proposition~\ref{semantics} shows that
the interval $[\mathsf{V}(S_7),\mathsf{V}(S)]$
contains $2^{\aleph_0}$ distinct varieties.
\end{proof}

Ren et al.~\cite[Example 5.7]{rjzl} showed that every word $\mathbf{w}_n$ is an isoterm for the flat semiring $M(abacdc)$.
Since $M(abacdc)$ contains a copy of $S_7$,
Theorem~\ref{thm1} immediately yields the following corollary.

\begin{cor}
The interval $[\mathsf{V}(S_7),\mathsf{V}(M(abacdc))]$ has cardinality $2^{\aleph_0}$.
\end{cor}

To prove that the interval $[\mathsf{V}(S_7),\mathsf{V}(B_2^1)]$ has cardinality $2^{\aleph_0}$,
by Theorem~\ref{thm1}, it is enough to show that every word $\mathbf{w}_n$ is an isoterm for $B_2^1$.
For this purpose, we first establish a sufficient condition for every $\mathbf{w}_n$ to be an isoterm for a given ai-semiring.

Let $\mathbf{w}$ be a word in $X^+$ and let $L$ be a subset of $X$.
For any word $\mathbf{w} \in X^{+}$ and any subset $L \subseteq X$, we denote by $\mathbf{w}(L)$ the word obtained from $\mathbf{w}$ by deleting all variables not belonging to $L$.
Equivalently, $\mathbf{w}(L) = \varphi_L(\mathbf{w})$,
where $\varphi_L \colon X\to X^{*}$ is a monoid substitution defined by
\[
\varphi_L(x) =
\begin{cases}
x, & x \in L,\\[4pt]
\varepsilon, & \text{otherwise}.
\end{cases}
\]
In particular,
if $L = c(\mathbf{u}) \setminus \{x\}$ for some $x \in X^{+}$,
we write $\mathbf{u}_x$ for $\mathbf{u}(L)$.

Suppose that $S$ is an ai-semiring with a multiplicative identity $1$.
For any proper subset $L$ of $c(\mathbf{v})$,
if $\mathbf{u}(L)\neq \varepsilon$, then
$\mathbf{u}(L) \preceq_{S} \mathbf{v}(L)$.
If $\mathbf{u}(L)=\varepsilon$,
we write $\varepsilon \preceq_{S} \mathbf{v}(L)$ to
mean that $1 \leq \varphi(\mathbf{v}(L))$ holds for every substitution $\varphi\colon X \to S$.

\begin{pro}\label{wns}
Let $S$ be an ai-semiring with a multiplicative identity $1$.
If $M_2 \in \mathsf{V}(S)$ and $xyxztz$ is an isoterm for $S$,
then every $\bw_n$ $(n \in \mathbb{N})$ is an isoterm for $S$.
\end{pro}
\begin{proof}
Let $n \in \mathbb{N}$.
To show that $\bw_n$ is an isoterm for $S$,
It suffices to prove that $\bw_n$ is minimal for $S$.
Suppose that $\bu\preceq_{S} \bw_n$ for some word $\bu$.
Since $M_2\in \mathsf{V}(S)$,
it follows that
\[
c(\bu)\subseteq c(\bw_n)=\{x_1, x_2, \ldots, x_{n+4}\}\cup\{y_0, y_1, \ldots, y_{n+1}\}.
\]

For any $1\leq i<j\leq n+4$, because $S$ contains a multiplicative identity $1$,
we have that $\bu(\{x_i,x_{j}\})\preceq_{S} \bw_n(\{x_i,x_{j}\})=x_ix_{j}$.
Since $xyxztz$ is an isoterm for $S$,
it follows that $x_ix_jx_i$ and $x_ix_j$ are both isoterms for $S$,
and so either $\bu(\{x_i,x_j\})=\varepsilon$ or $\bu(\{x_i,x_j\})=x_ix_j$.
If $\bu(\{x_i,x_j\})=\varepsilon$,
then $\varepsilon \preceq_{S} x_ix_j$, and so $x_i \preceq_{S} x_ix_jx_i$, a contradiction.
Thus $\bu(\{x_i, x_j\})=x_ix_j$.
Therefore,
\begin{equation}\label{eq3.3}
\mathbf{u}(\{x_1, x_2, \ldots, x_{n+4}\}) = x_1 x_2 \cdots x_{n+4}.
\end{equation}

Next, we show that each $y_i$ occurs exactly twice in $\mathbf{u}$ in the same order as in $\mathbf{w}_n$,
and moreover, the $x_j$'s appearing between the two occurrences of each $y_i$ are precisely those in $\mathbf{w}_n$.
Observe that
\begin{equation}\label{eq3.4}
\begin{aligned}
\mathbf{u}(\{y_0, y_1, x_i, x_3\}) &\preceq_S \mathbf{w}_n(\{y_0, y_1, x_i, x_3\}), \\
\mathbf{u}(\{y_n, y_{n+1}, x_{n+2}, x_j\}) &\preceq_S \mathbf{w}_n(\{y_n, y_{n+1}, x_{n+2}, x_j\}), \\
\mathbf{u}(\{y_k, y_{k+1}, x_{k+2}, x_{k+3}\}) &\preceq_S \mathbf{w}_n(\{y_k, y_{k+1}, x_{k+2}, x_{k+3}\}),
\end{aligned}
\end{equation}
where $i \in \{1, 2\}$, $j \in \{n+3, n+4\}$, and $k \in \{1, 2, \ldots, n-1\}$.

By~\eqref{eq3.3}, all left-hand sides in~\eqref{eq3.4} are nonempty words, while the corresponding right-hand sides are
$y_0 x_i y_0 y_1 x_3 y_1$, $y_n x_{n+2} y_n y_{n+1} x_j y_{n+1}$ and $y_k x_{k+2} y_k y_{k+1} x_{k+3} y_{k+1}$, respectively.
These words differ from $xyxztz$ only by names of the letters.
Since $xyxztz$ is an isoterm for $S$, the two sides of each inequality in~\eqref{eq3.4} must coincide.
Hence $\mathbf{u} = \mathbf{w}_n$, and so $\mathbf{w}_n$ is minimal for $S$, as required.
\end{proof}

By \cite[Lemma 5.3]{dgv}, every Zimin word $\mathbf{z}_n$ $(n \in \mathbb{N})$ is minimal for the ai-semiring $B_2^1$.
Using this, we now establish the following corollary.

\begin{cor}\label{coro26010305}
Every $\bw_n$ $(n \in \mathbb{N})$ is an isoterm for the ai-semiring $B_2^1$.
\end{cor}
\begin{proof}
It is easy to see that $B_2^1$ is an ai-semiring with the multiplicative identity $1$.
Also, $B_2^1$ contains a copy of $M_2$, and so $M_2\in \mathsf{V}(B_2^1)$.
By Proposition~\ref{wns}, it suffices to prove that
the word $xyxztz$ is an isoterm for $B_2^1$.

Now assume that $\bu\preceq_{B_2^1} xyxztz$ for some nonempty word $\bu$.
Since $M_2 \in \mathsf{V}(B_2^1)$, we obtain that $c(\bu) \subseteq c(xyxztz) = \{x, y, z, t\}$.
Since $B_2^1$ has a multiplicative identity,
it follows that $\bu_x \preceq_{B_2^1} yztz$.
Notice that $yztz$ differs from a subword of $\mathbf{z}_3$ only by names of the letters.
Hence $yztz$ is minimal for $B_2^1$,
and so either $\bu_x=\varepsilon$ or $\bu_x=yztz$.

If $\bu_x=\varepsilon$, consider the substitution $\varphi\colon X \to B_2^1$ defined by
$\varphi(y) = a$, $\varphi(z) = \varphi(t) = 1$.
Then $1 \leq \varphi(yztz) = a\cdot1\cdot 1\cdot 1=a$, a contradiction.
Hence $\bu_x = yztz$.
A similar argument shows that $\bu_z = xyxt$.
Consequently, $\bu$ must be $xyxztz$ or $xyzxtz$.

Now suppose that $\bu = xyzxtz$.
Let $\psi\colon X \to B_2^1$ be a substitution defined by
$\psi(x) = \psi(t) = a$ and $\psi(y) = \psi(z) = b$.
Then
\[
0=\psi(xyzxtz)=\psi(\bu) \leq \psi(xyzxtz) = ab,
\]
which is impossible.
Thus $\bu = xyxztz$, and so $xyxztz$ is an isoterm for $B_2^1$.
\end{proof}

\begin{proof}[Proof of Theorems~$\ref{maintheorem}$ and $\ref{thm26010401}$]
By Theorem~\ref{thm1} and Corollary~\ref{coro26010305},
the proof of Theorem~$\ref{maintheorem}$ is complete.
Recall that a word $\bw$ is an isoterm for a semigroup $S$ if
$S$ satisfies no nontrivial identity of the form $\bw \approx \bq$ for any word $\bq$.
By \cite[Lemma 3.3.32]{combinatorial-algebra}, $\bw$ is an isoterm for $S$
if and only if the semigroup $S(\bw)$ lies in $\mathsf{V}(S)$.
From the proof of Theorem~\ref{thm1}, one can deduce Theorem~\ref{thm26010401}.
\end{proof}

\section{Conclusion}
We have proved that the interval $[\mathsf{V}(S_7), \mathsf{V}(B_2^1)]$ contains $2^{\aleph_0}$ distinct varieties.
This result further elucidates the structure of the subvariety lattice of $\mathsf{V}(B_2^1)$
and illustrates the rich complexity inherent in this hierarchy.

To date, it is known that the subvariety lattice of $\mathsf{V}(\mathbf{Rel}(2))$
contains a chain of consecutive intervals, each of which has cardinality $2^{\aleph_0}$:
\[
[\mathbf{T},\mathsf{V}(S_7)],\quad
[\mathsf{V}(S_7),\mathsf{V}(B_2^1)],\quad
[\mathsf{V}(B_2^1),\mathsf{V}(\Sigma_7)],\quad\text{and}\quad
[\mathsf{V}(\Sigma_7),\mathsf{V}(\mathbf{Rel}(2))],
\]
where $\mathbf{T}$ denotes the trivial variety.

While there exist  finitely based finite semigroups whose varieties have
$2^{\aleph_0}$ distinct subvarieties (for instance, the semigroup $M(aba)$; see~\cite[Theorem 3.2]{jac:uncoutably}),
it had been unknown whether a finitely based finite ai-semiring
whose variety could have $2^{\aleph_0}$ distinct subvarieties.
Very recently we have discovered such an example with exactly three elements;
the details will be presented in a separate paper.

\subsection*{Acknowledgment}
The authors thank Professor Marcel Jackson for his insightful suggestions on this paper.

\section*{Declarations}
\noindent
\textbf{Ethical approval}
Not applicable.

\noindent
\textbf{Competing interests}
Not applicable.

\noindent
\textbf{Authors' contributions}
Not applicable.

\noindent
\textbf{Availability of data and materials}
Not applicable.

\noindent
\textbf{Funding}
The research of the second author was supported by National Natural Science Foundation of China
under grant no.~12371024 and no.~12571020.

\end{document}